\documentclass{amsart}
\usepackage{graphicx}
\usepackage{amssymb}
\usepackage{epstopdf}
\usepackage[all,cmtip]{xy}
\usepackage[hidelinks]{hyperref}

\usepackage{mathrsfs}

\newcommand{\fp}{\frak{p}}

\newcommand{\QQ}{\mathbb{Q}}
\newcommand{\ZZ}{\mathbb{Z}}
\newcommand{\sfC}{\mathsf{C}}

\newcommand{\sfK}{\mathsf{K}}
\newcommand{\sfT}{\mathsf{T}}

\newcommand{\sfW}{\mathsf{W}}

\def\im{\operatorname{im}}

\def\coker{\operatorname{coker}}
\def\ker{\operatorname{ker}}
\def\id{\operatorname{id}}

\def\pd{\operatorname{pd}}

\def\fd{\operatorname{fd}}
\def\id{\operatorname{id}}

\def\invlim{\varprojlim}

\def\Spec{\operatorname{Spec}}

\def\Hom{\operatorname{Hom}}

\def\Ext{\operatorname{Ext}}

\def\K{\mathsf{K}}
\def\Ktac{\mathsf{K}_{\operatorname{tac}}}
\def\DFtac{\mathsf{D}_{\operatorname{F-tac}}}

\def\Kpac{\mathsf{K}_{\operatorname{pac}}}

\def\D{\mathsf{D}}

\def\MF{\mathsf{MF}}
\def\HMF{\mathsf{HMF}}
\def\F{\mathsf{F}}
\def\HF{\mathsf{HF}}

\newcommand{\Inj}[1]{\mathsf{Inj}(#1)}
\newcommand{\prj}[1]{\mathsf{prj}(#1)}
\newcommand{\Prj}[1]{\mathsf{Prj}(#1)}
\newcommand{\Flat}[1]{\mathsf{Flat}(#1)}
\newcommand{\Cot}[1]{\mathsf{Cot}(#1)}
\newcommand{\FC}[1]{\mathsf{FlatCot}(#1)}

\newcommand{\Mod}[1]{\mathsf{Mod}(#1)}

\newcommand{\Cy}[2]{\operatorname{Z}_{#1}(#2)}

\newcommand{\GPdim}{\operatorname{Gpd}}
\newcommand{\GFdim}{\operatorname{Gfd}}

\newcounter{env}
\numberwithin{env}{section}

\theoremstyle{plain}
\newtheorem{thm}[env]{Theorem}          \newtheorem*{thm*}{Theorem}
\newtheorem{prp}[env]{Proposition}      \newtheorem*{prp*}{Proposition}
\newtheorem{cor}[env]{Corollary}        \newtheorem*{cor*}{Corollary}
\newtheorem{lem}[env]{Lemma}            \newtheorem*{lem*}{Lemma}
          \newtheorem*{cnj*}{Conjecture}
            \newtheorem*{fct*}{Fact}

\theoremstyle{definition}
\newtheorem{dfn}[env]{Definition}       \newtheorem*{dfn*}{Definition}
     \newtheorem*{con*}{Construction}
      \newtheorem*{obs*}{Observation}
\newtheorem{rmk}[env]{Remark}           \newtheorem*{rmk*}{Remark}
\newtheorem{exa}[env]{Example}          \newtheorem*{exa*}{Example}
         \newtheorem*{exe*}{Exercise}
         \newtheorem*{qst*}{Question}
            \newtheorem{stp*}{Setup}
            \newtheorem*{set*}{Setting}

\title[Matrix factorizations for self-orthogonal categories of modules]{Matrix factorizations for self-orthogonal categories of modules}
\date{December 3, 2019}                                           

\author{Petter Andreas Bergh} 
\address{Petter Andreas Bergh, Institutt for matematiske fag, NTNU, N-7491 Trondheim, Norway}
\email{petter.bergh@ntnu.no}
\author{Peder Thompson} 
\address{Peder Thompson, Institutt for matematiske fag, NTNU, N-7491 Trondheim, Norway}
\email{peder.thompson@ntnu.no}

\begin{document}
\maketitle
\begin{abstract}
For a commutative ring $S$ and self-orthogonal subcategory $\sfC$ of $\Mod{S}$, we consider matrix factorizations whose modules belong to $\sfC$. Let $f\in S$ be a regular element. If $f$ is $M$-regular for every $M\in \sfC$, we show there is a natural embedding of the homotopy category of $\sfC$-factorizations of $f$ into a corresponding homotopy category of totally acyclic complexes. Moreover, we prove this is an equivalence if $\sfC$ is the category of projective or flat-cotorsion $S$-modules. Dually, using divisibility in place of regularity, we observe there is a parallel equivalence when $\sfC$ is the category of injective $S$-modules.
\end{abstract}

\section*{Introduction}
\noindent
Matrix factorizations of a nonzero element $f$ in a regular local ring $Q$ were introduced by Eisenbud \cite{Eis80} and shown to correspond to maximal Cohen-Macaulay $Q/(f)$-modules; in turn Buchweitz \cite{Buc86} gave a relation between these and totally acyclic complexes of finitely generated projective $Q/(f)$-modules. Indeed, this correspondence can be described as an equivalence of triangulated categories,
\[\xymatrix{
\HMF(Q,f)\ar[r]^-{\simeq} & \Ktac(\prj{Q/(f)}),
}\]
where $\HMF(Q,f)$ is the homotopy category of matrix factorizations of $f$, and $\Ktac(\prj{Q/(f)})$ is the homotopy category of totally acyclic complexes of finitely generated projective $Q/(f)$-modules. In part, our goal is to develop the notion of matrix factorizations more generally---relative to a self-orthogonal category of modules---with an emphasis on extending this equivalence.

Let $S$ be a commutative ring, let $f\in S$, and let $\sfC$ be an additive subcategory of $\Mod{S}$, the category of $S$-modules. A linear factorization of $f$, defined by Dyckerhoff and Murfet \cite{DM13}, is a pair of $S$-modules $M_0$ and $M_1$ along with homomorphisms $d_1:M_1\to M_0$ and $d_0:M_0\to M_1$ satisfying $d_1d_0=f1^{M_0}$ and $d_0d_1=f1^{M_1}$. We define a \emph{$\sfC$-factorization of $f$} to be a linear factorization of $f$ such that $M_0,M_1\in \sfC$.  The homotopy category of $\sfC$-factorizations of $f$, denoted $\HF(\sfC,f)$, is the category whose objects are $\sfC$-factorizations of $f$ and whose morphisms are homotopy classes of the natural maps between $\sfC$-factorizations; see Section \ref{mfs}. Taking $\sfC$ to be the category of finitely generated projective modules over a regular local ring, one obtains the usual notion of matrix factorizations in \cite{Eis80}.

Set $R=S/(f)$. To relate a $\sfC$-factorization of $f$ to a suitable type of totally acyclic complex of $R$-modules, a natural setting to consider is when $\sfC$ is self-orthogonal, that is, $\Ext_S^i(M,M')=0$ for every $M,M'\in \sfC$ and $i\geq 1$. If $\sfC$ is self-orthogonal and $f\in S$ is $S$-regular and $M$-regular for every $M\in \sfC$, then the category $R\otimes_S \sfC$ is self-orthogonal---see Proposition \ref{change_of_rings}---in which case there is a natural notion of total acyclicity. Proposition \ref{acyclicW} thus relates $\sfC$-factorizations of $f$ to $R\otimes_S\sfC$-totally acyclic complexes. Here, for a self-orthogonal category $\sfW$ in $\Mod{R}$, a $\sfW$-totally acyclic complex is an acyclic complex of modules in $\sfW$ whose acyclicity is preserved by $\Hom_R(-,\sfW)$ and $\Hom_R(\sfW,-)$; this includes the usual notions of total acyclicity for complexes of projective or injective modules, and is a special case of that in \cite{CET18}.

In this setting, that is, if $\sfC$ is an additive self-orthogonal subcategory of $\Mod{S}$ and $f$ is $S$-regular and $M$-regular for every $M\in \sfC$, then we prove in Theorem \ref{fullyfaithful} that there is a full and faithful triangulated functor,
\[\xymatrix{
\sfT: \HF(\sfC,f)\ar[r] & \Ktac(R\otimes_S \sfC),
}\]
where $\Ktac(R\otimes_S\sfC)$ is the homotopy category of $R\otimes_S\sfC$-totally acyclic complexes. This embedding extends work of Bergh and Jorgensen; indeed, its proof is closely modelled on that of \cite[Theorem 3.5]{BJ16}, which is recovered by setting $\sfC=\prj{S}$. 

The functor $\sfT$ sends a $\sfC$-factorization of $f$ to a 2-periodic complex, see Proposition \ref{acyclicW}, and so we do not expect it to be an equivalence without additional assumptions on $S$ and $\sfC$. If $S$ is a regular local ring and $\sfC=\Prj{S}$ is the category of projective $S$-modules, then we show in Theorem \ref{mainthm} that there is a triangulated equivalence:
\[\xymatrix{
\HF(\Prj{S},f)\ar[r]^-{\simeq} & \Ktac(\Prj{R}).
}\]
Indeed, restricting to the subcategory of finitely generated projective modules, this is the equivalence due to Eisenbud \cite{Eis80} and Buchweitz \cite{Buc86} described above. 

Parallel to this development, we consider a dual situation in terms of divisibility. If $f$ is $S$-regular and $M$-divisible for every $M\in \sfC$, we observe in Theorem \ref{fullyfaithful_div} that there is an embedding $\HF(\sfC,f)\to \Ktac(\Hom_S(R,\sfC))$. In particular, since injective $S$-modules are divisible, we obtain an equivalence for $\sfC=\Inj{S}$, the category of injective $S$-modules, when $S$ is a regular local ring; see Theorem \ref{mainthm_inj}.

Another natural (torsion-free) self-orthogonal category to consider is $\FC{S}$, the category of flat-cotorsion $S$-modules; see Section \ref{equiv_fc}. We prove in Theorem \ref{mainthm_fc} that if $S$ is a regular local ring, then there is a triangulated equivalence:
\[\xymatrix{
\HF(\FC{S},f)\ar[r]^-{\simeq} & \Ktac(\FC{R}).
}\]
Here $\Ktac(\FC{R})$ is the homotopy category of acyclic complexes of flat-cotorsion $R$-modules such that for every flat-cotorsion $R$-module $F$, application of $\Hom_R(F,-)$ and $\Hom_R(-,F)$ preserves acyclicity.

In addition to the classic equivalence described above, Buchweitz gave in \cite{Buc86} an equivalence, assuming $S$ is a regular local ring, between the homotopy category of matrix factorizations of $f$ and the singularity category of $R$; this was proven explicitly by Orlov \cite{Orl04}. Along these lines, and as a consequence of the previous equivalence, we observe in Corollary \ref{purederivedcor} a triangulated equivalence,
\[\xymatrix{
\HF(\FC{S},f)\ar[r]^-{\simeq} & \DFtac(\Flat{R}),
}\]
where $\DFtac(\Flat{R})$ is the subcategory of the pure derived category of flat $R$-modules consisting of \textbf{F}-totally acyclic complexes. This category plays the role of the singularity category in the context of the pure derived category, in that it vanishes if and only if $R$ is regular; see \cite[Proposition 9.7]{Mur07} and \cite{MS11}.

\section{Self-orthogonal categories of modules}
\noindent
Throughout this paper, let $S$ be a commutative ring. The category of all $S$-modules is denoted $\Mod{S}$. Tacitly, we assume all subcategories of $\Mod{S}$ are full and closed under isomorphisms. We use standard homological notation throughout, and an $S$-complex means a chain complex of $S$-modules. 

Let $\Prj{S}$, $\Inj{S}$, $\Flat{S}$ denote the categories of projective, injective, and flat $S$-modules, respectively; $\prj{S}$ denotes the category of finitely generated projective $S$-modules. Let $\Cot{S}$ denote the category of cotorsion $S$-modules, that is, those $S$-modules $C$ such that $\Ext_S^1(F,C)=0$ for every flat $S$-module $F$. For brevity, write $\FC{S}=\Flat{S}\cap \Cot{S}$ for the category of flat-cotorsion $S$-modules.

\begin{dfn}
Let $\sfC$ be a subcategory of $\Mod{S}$. The category $\sfC$ is called \emph{self-orthogonal}\footnote{This differs from \cite{CET18}, where the term was used to refer to $\Ext^1$-orthogonality and is implied by the definition given here; our usage here agrees with what would be written as $\sfC\perp \sfC$ in \cite{SWSW08}.} if $\Ext_S^i(C,C')=0$ for all $C, C'\in \sfC$ and all $i\geq 1$.
\end{dfn}

\begin{exa}
Evidently both $\Prj{S}$ and $\Inj{S}$ are self-orthogonal. 

The category $\FC{S}$ is also self-orthogonal: Let $F$ and $F'$ be flat-cotorsion $S$-modules. If $P\to F$ is a projective resolution over $S$, then $\coker(d_i^P)$ is a flat $S$-module for $i\geq 1$, hence $\Ext_S^i(F,F')\cong \Ext_S^1(\coker(d_{i}^P),F')=0$ for all $i\geq 1$.
\end{exa}

\begin{dfn}
Let $M$ be an $S$-module, $f\in S$, and $\sfC$ be a subcategory of $\Mod{S}$. 

The element $f$ is \emph{$M$-regular} if $fx=0$ implies $x=0$ for each $x\in M$; $f$ is \emph{$\sfC$-regular} if $f$ is $M$-regular for every $M\in \sfC$. 

The element $f$ is \emph{$M$-divisible} if for every $x\in M$, there exists $y\in M$ with $fy=x$; $f$ is \emph{$\sfC$-divisible} if $f$ is $M$-divisible for every $M\in \sfC$.
\end{dfn}

\begin{exa}
Let $f\in S$ be an $S$-regular element. 

If $\sfC$ is a subcategory of $\Mod{S}$ contained in the category of torsion-free $S$-modules, then $f$ is $\sfC$-regular. In particular, $f$ is $\Flat{S}$-regular, $\FC{S}$-regular, and $\Prj{S}$-regular.

If $\sfC$ is a subcategory of $\Mod{S}$ contained in the category of divisible $S$-modules, then $f$ is $\sfC$-divisible. In particular, $f$ is $\Inj{S}$-divisible.
\end{exa}

Let $S\to R$ be a ring homomorphism and let $\sfC$ be a subcategory of $\Mod{S}$. The following subcategories of $\Mod{R}$ play a special role in this paper: 
\begin{align*}
R\otimes_S\sfC & =\{W\in \Mod{R} \mid W\cong R\otimes_S C\text{, for some $C\in \sfC$}\};\\
\Hom_S(R,\sfC) & =\{W\in \Mod{R} \mid W\cong \Hom_S(R,C)\text{, for some $C\in \sfC$}\}.
\end{align*}

\begin{rmk}\label{prj_inj_eq}
For any ring homomorphism $S\to R$, we have $R\otimes_S \Prj{S}\subseteq \Prj{R}$ and $\Hom_S(R,\Inj{S})\subseteq \Inj{R}$, see for example \cite[Proposition 2.3]{CJ09}; the former is an equality if the homomorphism is local, the second is an equality if the homomorphism is a surjection. The equality for projective modules uses that projective modules over a local ring are free. We justify the equality for injective modules here: Let $I$ be an injective $R$-module and let $I \to E_S(I)$ be its injective envelope as an $S$-module. Since the natural map $\Hom_S(R,I)\to I$ is an isomorphism, it follows that the induced injection $\Hom_S(R,I)\to \Hom_S(R,E_S(I))$ of $R$-modules is essential and splits, thus is an isomorphism. It follows that $I\cong \Hom_S(R,E_S(I))$.
\end{rmk}

For an $S$-module $M$, denote by $\pd_SM$, $\id_SM$, and $\fd_SM$ the projective, injective, and flat dimensions of $M$ over $S$.
\begin{rmk}\label{pdidfd}
Let $f\in S$ be an $S$-regular element, and set $R=S/(f)$.  If $P$ is a projective $R$-module, then $\pd_SP=1$ (see \cite[Part III, Theorem 3]{Kap72}); if $I$ is an injective $R$-module, then $\id_SI=1$ (see \cite[Theorem 202]{Kap70}). It thus follows that if $F$ is a flat $R$-module, then $\fd_SF=1$; this uses the fact that an $S$-module $M$ is flat if and only if its character dual $\Hom_{\ZZ}(M,\QQ/\ZZ)$ is injective. 
\end{rmk}

Part (i) of the next change of rings result is due to Rees \cite{Ree56}; part (iii) is dual. 
If $M$ is an $S$-module, $f\in S$, and $R=S/(f)$, it is often convenient to identify $R\otimes_SM\cong M/fM$ and $\Hom_S(R,M)\cong (0:_Mf)=\{x\in M \mid fx=0\}\subseteq M$. 
\begin{lem}\label{Rees}
Let $f\in S$ be an $S$-regular element and set $R=S/(f)$.

If $M$ is an $S$-module such that $f$ is $M$-regular and $N$ is an $R$-module, then
\begin{enumerate}
\item[(i)] $\Ext_S^{i+1}(N,M)\cong \Ext_R^i(N,R\otimes_S M)$ for all $i\geq 0$;
\item[(ii)] $\Ext_S^i(M,N)\cong \Ext_R^i(R\otimes_S M,N)$ for all $i\geq 0$.
\end{enumerate}

If $M$ is an $S$-module such that $f$ is $M$-divisible and $N$ is an $R$-module, then
\begin{enumerate}
\item[(iii)] $\Ext_S^{i+1}(M,N)\cong \Ext_R^i(\Hom_S(R,M),N)$ for all $i\geq 0$;
\item[(iv)] $\Ext_S^{i}(N,M)\cong \Ext_R^i(N,\Hom_S(R,M))$ for all $i\geq 0$.
\end{enumerate}
\end{lem}
\begin{proof}
(i) \& (ii): See Matsumura \cite[Lemma 2, p. 140]{Mat89} for a proof of these; (i) was originally shown by Rees \cite[Theorem 2.1]{Ree56}.

(iii): We give an argument dual to \cite[Theorem 2.1]{Ree56}, showing that the functor $E^i(-)=\Ext_S^{i+1}(M,-)$ is the $i$th right derived functor of $\Hom_R(\Hom_S(R,M),-)$. Apply $\Hom_S(-,N)$ to the short exact sequence 
\[\xymatrix{
0\ar[r] & \Hom_S(R,M)\ar[r] & M\ar[r]^{f} & M\ar[r] & 0
}\]
to obtain the following exact sequence
\[
\xymatrix@C=1em{
\Hom_S(M,N)\ar[r] & \Hom_S(\Hom_S(R,M),N)\ar[r] &  \Ext_S^1(M,N)\ar[r]^f &\Ext_S^1(M,N).
}\]
Since $fN=0$, we obtain $\Hom_S(M,N)=0$. Additionally, multiplication by $f$ on $M$ or $N$ induce the same map on $\Ext_S^1(M,N)$: also multiplication by $f$. As $fN=0$, this map must be 0, thus yielding 
$$\Ext_S^1(M,N)\cong \Hom_S(\Hom_S(R,M),N)\cong \Hom_R(\Hom_S(R,M),N).$$ 
Hence $E^0(-)\cong \Hom_R(\Hom_S(R,M),-)$. For any injective $R$-module $I$, we have $\id_SI=1$ by Remark \ref{pdidfd}, hence $E^i(I)=0$ for $i\geq 1$. Finally, for a short exact sequence $0\to N'\to N\to N''\to 0$ of $R$-modules, $\Hom_S(M,N'')=0$ and so there is a long exact sequence 
$$0\to E^0(N')\to E^0(N)\to E^0(N'')\to E^1(N')\to E^1(N)\to E^1(N'')\to \cdots,$$
and it follows that $E^i(-)$ is the $i$th right derived functor of $\Hom_R(\Hom_S(R,M),-)$ and thus is isomorphic to $\Ext_R^i(\Hom_S(R,M),-)$. 

(iv): Let $P$ be a projective resolution of $N$ over $R$; standard tensor--Hom adjunction yields $\Hom_S(R\otimes_R P,M)\cong \Hom_R(P,\Hom_S(R,M))$, and the desired isomorphism follows.
\end{proof}

\begin{prp}\label{change_of_rings}
Let $\sfC$ be a self-orthogonal subcategory of $\Mod{S}$, let $f\in S$ be $S$-regular, and set $R=S/(f)$. The following hold:
\begin{enumerate}
\item[(i)] If $f$ is $\sfC$-regular, then $R\otimes_S \sfC$ is self-orthogonal in $\Mod{R}$.
\item[(ii)] If $f$ is $\sfC$-divisible, then $\Hom_S(R,\sfC)$ is self-orthogonal in $\Mod{R}$.
\end{enumerate}
\end{prp}
\begin{proof}
(i): For $S$-modules $C,\ C'\in \sfC$ and $i\geq 0$, Lemma \ref{Rees}(ii) yields that $\Ext_R^i(R\otimes_S C,R\otimes_S C')\cong \Ext_S^i(C,R\otimes_S C')$.
It will therefore be enough to show that $\Ext_S^i(C,R\otimes_S C')=0$ for $i\geq 1$. As $f$ is $\sfC$-regular, there is an exact sequence
\[\xymatrix{
0\ar[r] & C'\ar[r]^{f} &C'\ar[r] &  R\otimes_S C'\ar[r] & 0.
}\]
Application of the functor $\Hom_S(C,-)$ yields a long exact sequence:
\[\xymatrix{
\cdots \ar[r] & \Ext_S^i(C,C')\ar[r]& \Ext_S^i(C,R\otimes_SC')\ar[r]& \Ext_S^{i+1}(C,C')\ar[r] & \cdots
}\]
By assumption, $\Ext_S^i(C,C')=0= \Ext_S^{i+1}(C,C')$ for $i\geq 1$, and it follows that $\Ext_S^i(C,R\otimes_SC')=0$ for $i\geq 1$. 

(ii): This is proved dually to part (i), using instead Lemma \ref{Rees}(iv) and the existence of an exact sequence
\[\xymatrix{
0\ar[r] & \Hom_S(R,C)\ar[r] &C\ar[r]^f &  C\ar[r] & 0
}\]
for each $C\in \sfC$.
\end{proof}

\section{$\sfC$-factorizations and total acyclicity}\label{mfs} 
\noindent
Let $f\in S$. Extending the classic notion of matrix factorizations \cite{Eis80}, Dyckerhoff and Murfet define \cite{DM13} a \emph{linear factorization of $f$} to be a $\ZZ/2\ZZ$-graded $S$-module $M=M_0\oplus M_1$ together with an $S$-linear differential $d:M\to M$ that is homogeneous of degree $1$ and satisfies $d^2=f 1^M$. 
We often write such a linear factorization as
$$(M,d)=(\xymatrix{M_1 \ar@<1mm>[r]^{d_1} & M_0\ar@<1mm>[l]^{d_0}}),$$
where $d_1d_0=f 1^{M_0}$ and $d_0d_1=f 1^{M_1}$.

A morphism $\alpha:(M,d)\to (M',d')$ of linear factorizations of $f$ is a degree 0 map which commutes with the differentials on $M$ and $M'$; it consists of maps $\alpha_i:M_i\to M_i'$, for $i=0,1$, making the following diagram commute:
\begin{align*}
\xymatrix{
M_1 \ar[r]^{d_1}\ar[d]^{\alpha_1}& M_0\ar[r]^{d_0}\ar[d]^{\alpha_0} & M_1\ar[d]^{\alpha_1}\\
M_1'\ar[r]^{d_1'} & M_0'\ar[r]^{d_0'} & M_1'
}
\end{align*}

\begin{dfn}
Let $\sfC$ be a subcategory of $\Mod{S}$. A {\em $\sfC$-factorization of $f$} is a linear factorization $(M,d)$ of $f$ such that $M_0,M_1\in \sfC$. 

Denote by $\F(\sfC,f)$ the category whose objects are $\sfC$-factorizations of $f$ and whose morphisms are those described above. 
\end{dfn}
In particular, if $\prj{S}$ is the category of finitely generated projective $S$-modules, then a $\prj{S}$-factorization of $f$ is the same as the usual notion of a \emph{matrix factorization} of $f$, that is, $\F(\prj{S},f)=\MF(S,f)$.

We say two morphisms $\alpha,\beta:(M,d)\to (M',d')$ of linear factorizations are {\em homotopic}, and write $\alpha\sim \beta$, if there exists homomorphisms $h_0:M_0\to M_1'$ and $h_1:M_1\to M_0'$ satisfying the usual homotopy conditions:
$$\alpha_0-\beta_0=h_1d_0+d_1'h_0\quad \text{and}\quad \alpha_1-\beta_1=h_0d_1+d_0'h_1.$$
From this, we define the associated {\em homotopy category of $\sfC$-factorizations of $f$}, denoted $\HF(\sfC,f)$, to be the homotopy category whose objects are the same as $\F(\sfC,f)$ and whose morphisms are homotopy classes of morphisms of $\sfC$-factorizations.

\begin{lem}\label{injective}
Let $(M,d)\in \F(\sfC,f)$. If $f$ is $M$-regular, then $d_1$ and $d_0$ are injective. If $f$ is $M$-divisible, then $d_1$ and $d_0$ are surjective.
\end{lem}
\begin{proof}
First assume $f$ is $M$-regular. The equality $d_0d_1=f1^{M_1}$ implies that for $x\in M_1$ with $d_1(x)=0$, we have $0=d_0d_1(x)=fx$. Since $f$ is $M$-regular, it follows that $x=0$, hence $d_1$ is injective. Injectivity of $d_0$ is proved similarly. 

Next assume $f$ is $M$-divisible. Let $x\in M_0$ be any element. Divisibility implies there exists $y\in M_0$ with $fy=x$. Since $d_1d_0=f1^{M_0}$, we have $d_1d_0(y)=fy=x$, hence $d_1$ is surjective. Surjectivity of $d_0$ is proved similarly.
\end{proof}

Given a category $\sfC$ of $S$-modules, the notions of (left and right) $\sfC$-totally acyclic complexes and (left and right) $\sfC$-Gorenstein modules were defined in \cite[Definition 1.1]{CET18}; in the case where $\sfC$ is self-orthogonal, these notions simplify to the following equivalent characterizations by \cite[Propositions 1.3 and 1.5]{CET18}. For an $S$-complex $T$, we set $\Cy{i}{T}=\ker(d_i^T)$ for each $i\in \ZZ$.
\begin{dfn}\label{tac_Gor}
Let $\sfC$ be a self-orthogonal category of $S$-modules. 
\begin{enumerate}
\item An $S$-complex $T$ is {\em $\sfC$-totally acyclic} if $T$ is acyclic, $T_i\in \sfC$ for $i\in \ZZ$, and for every $C\in \sfC$, the complexes $\Hom_S(T,C)$ and $\Hom_S(C,T)$ are also acyclic. 
\item An $S$-module $M$ is {\em $\sfC$-Gorenstein} if $M=\Cy{0}{T}$ for some $\sfC$-totally acyclic complex $T$. 
\end{enumerate}
The homotopy category of $\sfC$-totally acyclic complexes is denoted $\Ktac(\sfC)$.  If $\sfC$ is additive, then $\Ktac(\sfC)$ is triangulated. 
\end{dfn}

A $\Prj{S}$-Gorenstein module is called a {\em Gorenstein projective module} and an $\Inj{S}$-Gorenstein module is called a {\em Gorenstein injective module}; these are the standard notions appearing in the literature.

The next lemma is used below to relate cokernel modules of $\sfC$-factorizations to totally acyclic complexes.

\begin{lem}\label{coker}
Let $\sfC$ be a self-orthogonal subcategory of $\Mod{S}$, let $f\in S$ be $S$-regular and $\sfC$-regular, and set $R=S/(f)$. 
If $(M,d)\in \F(\sfC,f)$, then $\coker(d_1)$ and $\coker(d_0)$ are $R$-modules, and for any $C\in \sfC$ and $i\geq 1$ the following hold:
\begin{enumerate}
\item[(i)] $\Ext_R^i(R\otimes_SC,\coker(d_1))=0=\Ext_R^i(R\otimes_S C,\coker(d_0))$,
\item[(ii)] $\Ext_R^i(\coker(d_1),R\otimes_SC)=0=\Ext_R^i(\coker(d_0),R\otimes_SC)$.
\end{enumerate}
\end{lem}
\begin{proof}
We prove the statements for $\coker(d_1)$; proofs for $\coker(d_0)$ are similar. 

Note first that $\coker(d_1)$ is an $R$-module, since $f\coker(d_1)=0$; indeed, we have $fM_0\subseteq \im(d_1)$ as $f1^{M_0}=d_1d_0$, and so $f1^{M_0}$ induces the zero map on $\coker(d_1)$.

As $f$ is $\sfC$-regular, Lemma \ref{injective} yields an exact sequence
\begin{align}\label{star}
\xymatrix{
0\ar[r] & M_1\ar[r]^{d_1} & M_0\ar[r] & \coker(d_1)\ar[r]& 0.
}
\end{align}
Let $C\in \sfC$. Application of $\Hom_S(C,-)$  to the exact sequence $(\ref{star})$ yields a long exact sequence:
\[\xymatrix{
\cdots \ar[r]& \Ext_S^{i}(C,M_0)\ar[r]& \Ext_S^{i}(C,\coker(d_1))\ar[r]& \Ext_S^{i+1}(C,M_1)\ar[r]&\cdots
}\]
As $M_0$ and $M_1$ are in $\sfC$, we obtain that $\Ext_S^{i}(C,M_0)=0= \Ext_S^{i+1}(C,M_1)$ for $i\geq 1$, and hence $\Ext_S^{i}(C,\coker(d_1))=0$ for $i\geq 1$. Since $\coker(d_1)$ is an $R$-module, Lemma \ref{Rees}(ii) now yields $\Ext_R^i(R\otimes_SC,\coker(d_1))\cong \Ext_S^i(C,\coker(d_1))=0$ for $i\geq 1$. This gives (i).

For (ii), instead apply $\Hom_S(-,C)$ to the exact sequence $(\ref{star})$ to obtain a long exact sequence for $i\geq 1$:
\[\xymatrix{
\cdots \ar[r] & \Ext_S^{i}(M_1,C)\ar[r] & \Ext_S^{i+1}(\coker(d_1),C)\ar[r] & \Ext_S^{i+1}(M_0,C)\ar[r] & \cdots
}\]
As $M_0$ and $M_1$ are in $\sfC$, we obtain that $\Ext_S^i(M_1,C)=0=\Ext_S^{i+1}(M_0,C)$ for $i\geq 1$. It follows that $\Ext_S^{i+1}(\coker(d_1),C)=0$ for $i\geq 1$. Employing Lemma \ref{Rees}(i), we obtain $\Ext_R^i(\coker(d_1),R\otimes_SC)\cong \Ext_S^{i+1}(\coker(d_1),C)=0$ for all $i\geq 1$. 
\end{proof}

If $M$ is an $S$-module, $\alpha$ is an $S$-homomorphism, and $R=S/(f)$, then we set $\overline{M}=R\otimes_SM$ and $\overline{\alpha}=R\otimes_S\alpha$; context should make this clear.

\begin{prp}\label{acyclicW}
Let $\sfC$ be a subcategory of $\Mod{S}$, let $f\in S$ be $S$-regular and $\sfC$-regular, and set $R=S/(f)$. Let $(M,d)\in \F(\sfC,f)$. The $R$-sequence 
$$T^M:=\quad \xymatrix{\cdots \ar[r]^{\overline{d_0}} & \overline{M_1}\ar[r]^{\overline{ d_1}} & \overline{M_0} \ar[r]^{\overline{d_0}} & \overline{M_1} \ar[r]^{\overline{ d_1}} & \cdots}$$
is acyclic. If $\sfC$ is self-orthogonal, then $T^M$ is $R\otimes_S \sfC$-totally acyclic. 
\end{prp}
\begin{proof}
First, as $ d_1d_0=f1^{M_0}$ and $d_0 d_1=f1^{M_1}$, we have $\overline{ d_1}\; \overline{d_0}=0=\overline{d_0}\; \overline{ d_1}$ and so the sequence $T^M$ is a complex of $R$-modules.

We now show $T^M$ is acyclic. Let $x\in M_1$ such that $\overline{x}\in \ker(\overline{ d_1})$. It follows that $ d_1(x)\in fM_0$, whence there exists $y\in M_0$ such that $d_1(x)=fy$. As $fy= d_1d_0(y)$, it follows that $d_1(x)= d_1d_0(y)$, hence $ d_1(x-d_0(y))=0$. Injectivity of $d_1$, see Lemma \ref{injective}, implies that $x=d_0(y)$. Hence $\overline{d_0}(\overline{y})=\overline{x}$, and so $H_{2i+1}(T^M)=0$ for every $i\in \ZZ$. A similar argument (using injectivity of $d_0$) yields $H_{2i}(T^M)=0$ for every $i\in \ZZ$, thus proving the complex $T^M$ is acyclic.

Multiplication by $f$ on the exact sequence $0\to M_1\xrightarrow{d_1} M_0\to \coker(d_1)\to 0$, along with the snake lemma, yields an exact sequence 
\[\xymatrix{
\coker(d_1)\ar[r]^f & \coker(d_1)\ar[r] & \coker(\overline{d_1}) \ar[r] & 0.
}\]
Since $\coker(d_1)$ is an $R$-module (see Lemma \ref{coker}), this implies $\coker(\overline{d_1})\cong\coker(d_1)$; similarly, $\coker(\overline{d_0})\cong\coker(d_0)$. Acyclicity of $T^M$ gives $\Cy{2i}{T^M}\cong\coker(d_0)$ and $\Cy{2i+1}{T^M}\cong\coker(d_1)$ for every $i\in \ZZ$. 

Fix $C\in \sfC$. To verify the complexes $\Hom_R(T^M,R\otimes_S C)$ and $\Hom_R(R\otimes_S C,T^M)$ are acyclic, it suffices to show that the exact sequences
\[\xymatrix{
0\ar[r] & \coker(d_0) \ar[r] & \overline{M_0} \ar[r] & \coker(d_1) \ar[r] & 0
}\]
and
\[\xymatrix{
0\ar[r] & \coker(d_1) \ar[r] & \overline{M_1} \ar[r] & \coker(d_0) \ar[r] & 0
}\]
remain exact upon application of $\Hom_R(R\otimes_S C,-)$ and $\Hom_R(-,R\otimes_S C)$. This follows from Lemma \ref{coker}. Therefore, as $R\otimes_S \sfC$ is self-orthogonal by Proposition \ref{change_of_rings}, we obtain that $T^M$ is $R\otimes_S C$-totally acyclic.
\end{proof}

We have the next dual results involving divisibility:
\begin{lem}\label{ker}
Let $\sfC$ be a self-orthogonal subcategory of $\Mod{S}$, let $f\in S$ be $S$-regular and $\sfC$-divisible, and set $R=S/(f)$. 
If $(M,d)\in \F(\sfC,f)$, then $\ker(d_1)$ and $\ker(d_0)$ are $R$-modules, and for any $C\in \sfC$ and $i\geq 1$ the following hold:
\begin{enumerate}
\item[(i)] $\Ext_R^i(\Hom_S(R,C),\ker(d_1))=0=\Ext_R^i(\Hom_S(R,C),\ker(d_0))$,
\item[(ii)] $\Ext_R^i(\ker(d_1),\Hom_S(R,C))=0=\Ext_R^i(\ker(d_0),\Hom_S(R,C))$.
\end{enumerate}
\end{lem}
\begin{proof}
Dual to the proof of Lemma \ref{coker}; use instead Lemma \ref{Rees}(iii,iv).
\end{proof}

\begin{prp}\label{acyclicW_div}
Let $\sfC$ be a subcategory of $\Mod{S}$, let $f\in S$ be $S$-regular and $\sfC$-divisible, and set $R=S/(f)$. Let $(M,d)\in \F(\sfC,f)$. The $R$-sequence 
$$\widetilde{T}^M:=\quad \xymatrix@C=1.5em{\cdots \ar[rr]^(.35){(d_0)_*} && \Hom_S(R,{M_1})\ar[rr]^{(d_1)_*} && \Hom_S(R,{M_0}) \ar[rr]^(.65){(d_0)_*} &&\cdots}$$
is acyclic. If $\sfC$ is self-orthogonal, then $\widetilde{T}^M$ is $\Hom_S(R,\sfC)$-totally acyclic.
\end{prp}
\begin{proof}
Dual to the proof of Proposition \ref{acyclicW}; use instead Lemma \ref{ker}.
\end{proof}

\section{A full and faithful functor}
\noindent
Let $\sfC$ be a self-orthogonal subcategory of $\Mod{S}$. We denote by $\sfK(\sfC)$ the homotopy category of $\sfC$, whose objects are complexes of modules in $\sfC$ and morphisms are homotopy classes of degree zero chain maps. Further, we consider the full subcategory $\Ktac(\sfC)$ whose objects are the $\sfC$-totally acyclic complexes in $\sfK(\sfC)$. 

\begin{prp}\label{functor}
Let $\sfC$ be an additive self-orthogonal subcategory of $\Mod{S}$, let $f\in S$ be $S$-regular and $\sfC$-regular, and set $R=S/(f)$. There is a triangulated functor 
\[\xymatrix{
\sfT:\HF(\sfC,f)\ar[r] & \Ktac(R\otimes_S\sfC)
}\]
defined, in notation from Proposition \ref{acyclicW}, as $\sfT(M,d)=T^M$ and $\sfT([\alpha])=[\overline{\alpha}]$.
\end{prp}
\begin{proof}
Let $[\alpha],[\beta]:(M,d)\to (M',d')$ be morphisms in $\HF(\sfC,f)$.  Set $T^M$ and $T^{M'}$ as the complexes constructed in Proposition \ref{acyclicW} and associated to $M$ and $M'$, respectively. Define $\overline{\alpha},\overline{\beta}:T^M\to T^{M'}$ as the evident 2-periodic chain maps induced by $\alpha$ and $\beta$. If $[\alpha]=[\beta]$, then there is a homotopy $h$ from $\alpha$ to $\beta$; this induces a 2-periodic homotopy $\overline{h}$ from $\overline{\alpha}$ to $\overline{\beta}$, implying that $[\overline{\alpha}]=[\overline{\beta}]$ in $\Ktac(R\otimes_S\sfC)$. Notice that as $\overline{1^M}=1^{T^M}$, if $[\alpha]=[1^M]$, then $[\overline{\alpha}]=[1^{T^M}]$.

Define a functor $\sfT:\HF(\sfC,f)\to \Ktac(R\otimes_S\sfC)$ as follows: For an object $(M,d)$, set $\sfT(M,d)=T^M$, and for a morphism $[\alpha]:(M,d)\to (M',d')$, set $\sfT([\alpha])=[\overline{\alpha}]$. The above remarks justify that $\sfT$ is well-defined on both objects and morphisms, that $\sfT$ preserves identities, and that $\sfT$ preserves compositions by the following equalities:
$$\sfT([\alpha])\sfT([\beta])=[\overline{\alpha}][\overline{\beta}]=[(\overline{\alpha})(\overline{\beta})]=[\overline{\alpha\beta}]=\sfT([\alpha\beta]).$$
Moreover, the functor $\sfT$ respects the triangulated structures, that is, $\sfT$ is additive, $\sfT((M,d)[1])=T^{M[1]}=T^M[1]=\sfT((M,d))[1]$, and $\sfT$ preserves exact triangles.
\end{proof}

\begin{lem}\label{lift}
Let $\sfC$ be a self-orthogonal subcategory of $\Mod{S}$, let $f\in S$ be $\sfC$-regular, and set $R=S/(f)$.  If $M,M'\in \sfC$ and $\varphi\in \Hom_R(\overline{M},\overline{M'})$, then there exists $\psi\in \Hom_S(M,M')$ such that $\overline{\psi}=\varphi$.
\end{lem}
\begin{proof}
Let $\varphi:\overline{M}\to \overline{M'}$ be an $R$-homomorphism. There is an exact sequence
\[\xymatrix{
0\ar[r] & M'\ar[r]^f & M'\ar[r]^{\pi'} & \overline{M'}\ar[r] & 0.
}\]
As $\Ext_S^1(M,M')=0$, we obtain an exact sequence
\[\xymatrix{
0\ar[r] & \Hom_S(M,M')\ar[r] & \Hom_S(M,M')\ar[r] & \Hom_S(M,\overline{M'})\ar[r] & 0.
}\]
Let $\pi:M\to \overline{M}$ be the canonical quotient map. The map $\varphi\pi\in \Hom_S(M,\overline{M'})$ lifts to a map $\psi\in \Hom_S(M,M')$ such that $\pi'\psi=\varphi\pi$, that is, $\overline{\psi}=\varphi$.
\end{proof}

The following arguments to show that $\sfT$ is full and faithful closely follow those given in \cite{BJ16}, put into the more general setting of totally acyclic complexes from \cite{CET18}.

\begin{prp}\label{faithful}
The functor $\sfT$ in Proposition \ref{functor} is faithful.
\end{prp}
\begin{proof}
Set $\sfW=R\otimes_S \sfC$. Let $[\alpha]:M\to M'$ be a morphism in $\HF(\sfC,f)$ such that $\sfT([\alpha])$ is the zero morphism in $\Ktac(\sfW)$. Our goal is to show $[\alpha]=[0]$, that is, $\alpha$ is null homotopic in $\F(\sfC,f)$. Write $\alpha:M\to M'$ as:

\[\xymatrix{
M_1 \ar[r]^{d_1}\ar[d]^{\alpha_1} & M_0 \ar[r]^{d_0}\ar[d]^{\alpha_0} & M_1\ar[d]^{\alpha_1}\\
M_1' \ar[r]^{d_1'} & M_0' \ar[r]^{d_0'} & M_1'
}\]

Let $\overline{\alpha}:\sfT(M,d)\to \sfT(M',d')$ denote the 2-periodic chain map induced by $\alpha$. The assumption $\sfT([\alpha])=[0]$ in $\Ktac(\sfW)$ implies that $\overline{\alpha}$ is null homotopic (i.e., $\overline{\alpha}\sim 0$). Let $\sigma$ be a null homotopy for $\overline{\alpha}$; notice, however, that $\sigma$ need not be 2-periodic. We have the following diagram:

\[\xymatrix@=3em{
\cdots \ar[r] &\overline{M_1} \ar[r]^{\overline{d_1}} \ar[d]^{\overline{\alpha_1}}& \overline{M_0} \ar[r]^{\overline{d_0}} \ar[d]^{\overline{\alpha_0}}\ar[dl]^{\sigma_2}& \overline{M_1} \ar[r]^{\overline{d_1}}\ar[d]^{\overline{\alpha_1}}\ar[dl]^{\sigma_1} & \overline{M_0} \ar[r]^{\overline{d_0}}\ar[d]^{\overline{\alpha_0}}\ar[dl]^{\sigma_0} & \overline{M_1}\ar[d]^{\overline{\alpha_1}}\ar[dl]^{\sigma_{-1}}\ar[r]&\cdots\\
\cdots \ar[r] &\overline{M_1'} \ar[r]_{\overline{d_1'}} & \overline{M_0'} \ar[r]_{\overline{d_0'}} &\overline{M_1'} \ar[r]_{\overline{d_1'}} & \overline{M_0'} \ar[r]_{\overline{d_0'}} & \overline{M_1'}\ar[r]&\cdots
}\]

In particular, we have the following relations (coming from degrees 1 and 2):
\begin{align}
\overline{\alpha_1}&=\overline{d_0'}\sigma_1+\sigma_0\overline{d_1},\label{alpha1rels}\\
\overline{\alpha_0}&=\overline{d_1'}\sigma_2+\sigma_1\overline{d_0}.\label{alpha0rels}
\end{align}

Lemma \ref{lift} yields $S$-module homomorphism liftings $h_{2i}:M_0\to M_1'$ of $\sigma_{2i}$ and $h_{2i+1}:M_1\to M_0'$ of $\sigma_{2i+1}$ for $i\in \ZZ$. The exact sequence $0\to M_1'\xrightarrow{f}M_1'\xrightarrow{\pi} \overline{M_1'}\to 0$ induces an exact sequence:
\[\xymatrix{
0\ar[r] & \Hom_S(M_1,M_1')\ar[r]^f & \Hom_S(M_1,M_1')\ar[r]^{\pi_*} & \Hom_S(M_1,\overline{M_1'})\ar[r] & 0,
}\]
where $\pi_*=\Hom_S(M_1,\pi)$. Since $\alpha_1-d_0'h_1-h_0d_1\in \ker(\pi_*)$ by (\ref{alpha1rels}), one obtains a map $\beta_1\in \Hom_S(M_1,M_1')$ such that $f\beta_1=\alpha_1-d_0'h_1-h_0d_1$. Similarly, using instead (\ref{alpha0rels}), one obtains $\beta_2\in \Hom_S(M_0,M_0')$ such that $f\beta_2=\alpha_0-d_1'h_2-h_1d_0$.

Define $s_1=h_1+d_1'\beta_1$. We claim that $(h_0,s_1)$ is a null homotopy of $\alpha:M\to M'$. 
First, we have:
\begin{align*}
d_0's_1+h_0d_1&=d_0'(h_1+d_1'\beta_1)+h_0d_1\\
&=d_0'h_1+d_0'd_1'\beta_1+h_0d_1\\
&=d_0'h_1+f\beta_1+h_0d_1\\
&=d_0'h_1+\alpha_1-d_0'h_1-h_0d_1+h_0d_1\\
&=\alpha_1.
\end{align*}
Next, precomposing the equality $f\beta_1=\alpha_1-d_0'h_1-h_0d_1$ with $d_0$ gives:
\begin{align*}
f\beta_1d_0&=(\alpha_1-d_0'h_1-h_0d_1)d_0\\
&=\alpha_1d_0-d_0'h_1d_0-h_0f\\
&=d_0'\alpha_0-d_0'h_1d_0-h_0f\\
&=d_0'(\alpha_0-h_1d_0)-h_0f\\
&=d_0'(f\beta_2+d_1'h_2)-h_0f\\
&=fd_0'\beta_2+d_0'd_1'h_2-fh_0\\
&=f(d_0'\beta_2+h_2-h_0).
\end{align*}
As $f$ is $M_1'$-regular, this yields
\begin{align}\label{beta1d0}
\beta_1d_0=d_0'\beta_2+h_2-h_0.
\end{align}
We therefore obtain:
\begin{align*}
d_1'h_0+s_1d_0&=d_1'h_0+(h_1+d_1'\beta_1)d_0\\
&=d_1'h_0+h_1d_0+d_1'\beta_1d_0\\
&=d_1'h_0+h_1d_0+d_1'(d_0'\beta_2+h_2-h_0)\text{, by (\ref{beta1d0}),}\\
&=d_1'h_0+h_1d_0+f\beta_2+d_1'h_2-d_1'h_0\\
&=h_1d_0+\alpha_0-d_1'h_2-h_1d_0+d_1'h_2\\
&=\alpha_0.
\end{align*}
Hence $\alpha:M\to M'$ is homotopic to 0, that is, $[\alpha]=[0]$ in $\HF(\sfC,f)$.
\end{proof}

\begin{prp}\label{full}
The functor $\sfT$ in Proposition \ref{functor} is full.
\end{prp}
\begin{proof}
Set $\sfW=R\otimes_S \sfC$. Let $(M,d)$ and $(M',d')$ be objects in $\HF(\sfC,f)$ and suppose $\overline{\alpha}:\sfT(M,d)\to \sfT(M',d')$ is a degree $0$ chain map, not necessarily 2-periodic, that represents a morphism $[\overline{\alpha}]$ in $\Ktac(\sfW)$; in particular, we have a commutative diagram:
\[\xymatrix{
\cdots \ar[r] & \overline{M_1} \ar[r]^{\overline{d_1}}\ar[d]^{\overline{\alpha}_3} & \overline{M_0}\ar[r]^{\overline{d_0}}\ar[d]^{\overline{\alpha}_2} & \overline{M_1} \ar[r]^{\overline{d_1}}\ar[d]^{\overline{\alpha}_1} & \overline{M_0}\ar[r]\ar[d]^{\overline{\alpha}_0} & \cdots\\
\cdots \ar[r] & \overline{M_1'} \ar[r]^{\overline{d_1'}} & \overline{M_0'}\ar[r]^{\overline{d_0'}} & \overline{M_1'} \ar[r]^{\overline{d_1'}} & \overline{M_0'}\ar[r] & \cdots
}\]

By Lemma \ref{lift}, for $i\in \ZZ$ we can lift $\overline{\alpha}_{2i}$ to $\alpha_{2i}:M_0\to M_0'$ and $\overline{\alpha}_{2i+1}$ to $\alpha_{2i+1}:M_1\to M_1'$.  In particular, we obtain the following diagram that commutes modulo $f$:
\[\xymatrix{
M_0 \ar[r]^{d_0}\ar[d]^{\alpha_2} & M_1 \ar[r]^{d_1}\ar[d]^{\alpha_1} & M_0\ar[d]^{\alpha_0}\\
M_0'\ar[r]^{d_0'} & M_1'\ar[r]^{d_1'} & M_0'
}\]
The exact sequence $0\to M_1'\xrightarrow{f}M_1'\xrightarrow{\pi} \overline{M_1'}\to 0$ induces an exact sequence
\[\xymatrix{
0\ar[r] & \Hom_S(M_0,M_1')\ar[r]^{f} & \Hom_S(M_0,M_1')\ar[r]^{\pi_*} & \Hom_S(M_0,\overline{M_1'})\ar[r] & 0.
}\]
Since $\alpha_1d_0-d_0'\alpha_2\in \ker(\pi_*)$, there exists a map $\sigma_0\in \Hom_S(M_0,M_1')$ such that 
\begin{align}\label{sigma0}
\alpha_1d_0-d_0'\alpha_2=f\sigma_0.
\end{align}
Similarly, there exists $\sigma_1\in \Hom_S(M_1,M_0')$ such that
\begin{align}\label{sigma1}
\alpha_0d_1-d_1'\alpha_1=f\sigma_1.
\end{align}

We now define new maps in order to construct a morphism in $\F(\sfC,f)$; define
\begin{align*}
\gamma_0&=\alpha_0+d_1'\sigma_0, \text{ and}\\
\gamma_1&=\alpha_1+d_0'\sigma_1+\sigma_0d_1.
\end{align*}
We aim to verify that the following diagram is commutative:
\begin{align}
\label{candidate}\xymatrix{
M_0 \ar[r]^{d_0}\ar[d]^{\gamma_0} & M_1 \ar[r]^{d_1}\ar[d]^{\gamma_1} & M_0\ar[d]^{\gamma_0}\\
M_0'\ar[r]^{d_0'} & M_1'\ar[r]^{d_1'} & M_0'
}
\end{align}
The equality (\ref{sigma1}), along with $d_1d_0=f1^{M_0}$ and $d_0'd_1'=f1^{M_1'}$, imply
$$fd_0'\sigma_1d_0=d_0'(\alpha_0d_1-d_1'\alpha_1)d_0=fd_0'\alpha_0-f\alpha_1d_0,$$
and so as $f$ is $M_1'$-regular, we have
\begin{align}\label{d0sigma1d0}
d_0'\sigma_1d_0=d_0'\alpha_0-\alpha_1d_0.
\end{align}
First we verify the left square of (\ref{candidate}) commutes:
\begin{align*}
\gamma_1d_0&=(\alpha_1+d_0'\sigma_1+\sigma_0d_1)d_0\\
&=\alpha_1d_0+d_0'\sigma_1d_0+f\sigma_0\\
&=\alpha_1d_0+(d_0'\alpha_0-\alpha_1d_0)+f\sigma_0,\quad\text{by (\ref{d0sigma1d0}),}\\
&=d_0'\alpha_0+f\sigma_0\\
&=d_0'\alpha_0+d_0'd_1'\sigma_0\\
&=d_0'(\alpha_0+d_1'\sigma_0)\\
&=d_0'\gamma_0.
\end{align*}
Next we verify the right square of (\ref{candidate}) commutes: 
\begin{align*}
d_1'\gamma_1&=d_1'(\alpha_1+d_0'\sigma_1+\sigma_0d_1)\\
&=d_1'\alpha_1+f\sigma_1+d_1'\sigma_0d_1\\
&=d_1'\alpha_1+\alpha_0d_1-d_1'\alpha_1+d_1'\sigma_0d_1,\quad\text{by (\ref{sigma1}),}\\
&=\alpha_0d_1+d_1'\sigma_0d_1\\
&=(\alpha_0+d_1'\sigma_0)d_1\\
&=\gamma_0d_1.
\end{align*}
Thus $\gamma=(\gamma_0,\gamma_1)$ is a morphism $(M,d)\to (M',d')$ in $\F(\sfC,f)$. 

We next claim $\overline{\alpha}\sim \overline{\gamma}$, i.e., that $\sfT([\gamma])=[\overline{\alpha}]$. We start with the following diagram (displaying homological degrees $3$ to $-1$):

\[\xymatrix@=3em{
\cdots \ar[r] & \overline{M_1}\ar[r]^{\overline{d_1}} \ar[d]^(.4){\overline{\gamma_1}-\overline{\alpha}_3} & \overline{M_0} \ar[r]^{\overline{d_0}}\ar[d]^(.4){\overline{\gamma_0}-\overline{\alpha}_2} & \overline{M_1} \ar[r]^{\overline{d_1}}\ar[d]^(.4){\overline{\gamma_1}-\overline{\alpha}_1}\ar[dl]^{\overline{\sigma_1}} & \overline{M_0}\ar[r]^{\overline{d_0}}\ar[d]^(.4){\overline{\gamma_0}-\overline{\alpha}_0}\ar[dl]^{\overline{\sigma_0}} & \overline{M_1} \ar[r] \ar[d]^(.4){\overline{\gamma_1}-\overline{\alpha}_{-1}} & \cdots \\
\cdots \ar[r] & \overline{M_1'} \ar[r]_{\overline{d_1'}} & \overline{M_0'} \ar[r]_{\overline{d_0'}} & \overline{M_1'} \ar[r]_{\overline{d_1'}} & \overline{M_0'} \ar[r]_{\overline{d_0'}} & \overline{M_1'} \ar[r] & \cdots
}\]

Evidently, $\overline{\sigma_1}$ and $\overline{\sigma_0}$ give the start of a homotopy in degree $1$:
$$\overline{\gamma_1}-\overline{\alpha}_1=\overline{\alpha}_1+\overline{d_0'}\overline{\sigma_1}+\overline{\sigma_0}\overline{d_1}-\overline{\alpha}_1=\overline{d_0'}\overline{\sigma_1}+\overline{\sigma_0}\overline{d_1}.$$
Note that the subcategory $\sfW$ is self-orthogonal by Proposition \ref{change_of_rings}. As $\sfT(M,d)$ and $\sfT(M',d')$ are $\sfW$-totally acyclic complexes, the arguments in \cite[Appendix]{CT19} show that we may extend $\overline{\sigma_1}$ and $\overline{\sigma_0}$ to a null homotopy of the displayed morphism, giving $\overline{\gamma}\sim \overline{\alpha}$. Indeed, extending the homotopy to the left is done by the proof of \cite[Proposition A.3]{CT19}, with $\sfW$-total acyclicity of $\sfT(M',d')$ standing in for the assumptions in \emph{loc.\ cit.} (and \cite[Lemma 2.4]{CFH06} in place of \cite[Lemma 2.5]{CFH06}); extending the homotopy to the right uses the dual proof for \cite[Proposition A.1]{CT19}. It follows that $\sfT([\gamma])=[\overline{\alpha}]$ hence $\sfT$ is full.
\end{proof}

The following recovers \cite[Theorem 3.5]{BJ16} when one takes $\sfC=\prj{S}$.

\begin{thm}\label{fullyfaithful}
Let $\sfC$ be an additive self-orthogonal subcategory of $\Mod{S}$, let $f\in S$ be $S$-regular and $\sfC$-regular, and set $R=S/(f)$. The triangulated functor $\sfT:\HF(\sfC,f)\to \Ktac(R\otimes_S \sfC)$ is full and faithful.
\end{thm}
\begin{proof}
Combine Propositions \ref{functor}, \ref{faithful}, and \ref{full}.
\end{proof}

In fact, the results of this section have dual statements involving divisibility. In summary, one can show the following:
\begin{thm}\label{fullyfaithful_div}
Let $\sfC$ be an additive self-orthogonal subcategory of $\Mod{S}$, let $f\in S$ be $S$-regular and $\sfC$-divisible, and set $R=S/(f)$. There is a triangulated functor $\widetilde{\sfT}:\HF(\sfC,f)\to \Ktac(\Hom_S(R,\sfC))$ that is full and faithful.
\end{thm}
\begin{proof}
One first notices that a version of Proposition \ref{functor} holds, by defining a functor $\widetilde{\sfT}$ using Proposition \ref{acyclicW_div}. Then using a dual version of Lemma \ref{lift}, one can establish analogues of Propositions \ref{faithful} and \ref{full}.
\end{proof}

\section{Equivalences for projective and injective factorizations}\label{equiv_prj}
\noindent
In this section, we consider $\Prj{S}$- and $\Inj{S}$-factorizations, referred to as projective and injective factorizations, respectively. Our goal here is to show that if $S$ is a regular local ring, $f\in S$ is nonzero, and $R=S/(f)$, then projective factorizations of $f$ correspond to Gorenstein projective $R$-modules; this can be considered as an extension of the classic bijection \cite[Corollary 6.3]{Eis80} between matrix factorizations (having no trivial direct summand) and maximal Cohen-Macaulay $R$-modules (having no free direct summand). Dually, we observe a correspondence between injective factorizations of $f$ and Gorenstein injective $R$-modules.

If one considers $\prj{S}$ in place of $\Prj{S}$ in the next result, then the classic proof, as in \cite{Eis80}, uses the Auslander--Buchsbaum formula. However, we use an approach here that does not require the modules to be finitely generated. 

\begin{prp}\label{gptomf}
Assume $S$ is a regular local ring, let $f\in S$ be nonzero, and set $R=S/(f)$. 
If $M$ is a Gorenstein projective $R$-module, then there exists a projective factorization $(P,d)\in \F(\Prj{S},f)$ with $\coker(d_1)=M$.
\end{prp}

\begin{proof}
Let $M$ be a Gorenstein projective $R$-module. As $fM=0$, a result of Bennis and Mahdou \cite[Theorem 4.1]{BM10} yields $\GPdim_SM=\GPdim_RM+1=1$, where $\GPdim$ denotes Gorenstein projective dimension. As $S$ is regular, $M$ has finite projective dimension over $S$, thus by \cite[4.4.7]{Chr00} we have $\pd_SM=\GPdim_SM=1$.

Now a standard construction yields a projective factorization of $f$ which corresponds to $M$: First choose a projective resolution $P$ of $M$ over $S$ having the form $0\to P_1\xrightarrow{d_1}P_0\to M\to 0$. Application of $\Hom_S(P_0,-)$ to this sequence gives an exact sequence:
\[\xymatrix{
0\ar[r] & \Hom_S(P_0,P_1)\ar[r] & \Hom_S(P_0,P_0)\ar[r] & \Hom_S(P_0,M)\ar[r] & 0.
}\]
As $fM=0$, the map $f1^{P_0}$ is sent to $0$, hence this sequence shows there exists a map $d_0:P_0\to P_1$ such that $d_1d_0=f1^{P_0}$. Further, $d_1(d_0d_1)=f1^{P_0}d_1=d_1(f1^{P_1})$, and since $d_1$ is injective this implies that $d_0d_1=f1^{P_1}$. It follows that $(P,d)$ is a projective factorization of $f$ such that $\coker (d_1)=M$.
\end{proof}

\begin{thm}\label{mainthm}
Assume $S$ is a regular local ring, let $f\in S$ be nonzero, and set $R=S/(f)$. 
There is a triangulated equivalence
\[
\xymatrix{
\sfT: \HF(\Prj{S},f)\ar[r]^-{\simeq} & \Ktac(\Prj{R})
}\]
given by the functor from Proposition \ref{functor}.
\end{thm}
\begin{proof}
The triangulated functor $\sfT$ given in Proposition \ref{functor}, applied to $\sfC=\Prj{S}$, is full and faithful by Theorem \ref{fullyfaithful}. Also note that $R\otimes_S \Prj{S}=\Prj{R}$ (see Remark \ref{prj_inj_eq}) and so the functor $\sfT$ has the claimed codomain.  It remains to show that $\sfT$ is essentially surjective. Let $T\in \Ktac(\Prj{R})$. Then $\Cy{0}{T}$ is a Gorenstein projective $R$-module. By Proposition \ref{gptomf} there is a $\Prj{S}$-factorization $(P,d)$ such that $\coker(d_1)=\Cy{0}{T}$. 

We argue that $\sfT(P,d)$ is homotopic to $T$.  Notice that $\Cy{0}{\sfT(P,d)}=\Cy{0}{T}$ by construction. There exists a degree 0 chain map $\phi:\sfT(P,d)\to T$ that lifts the identity map $\Cy{0}{\sfT(P,d)}\xrightarrow{=} \Cy{0}{T}$ by \cite[Lemma 3.1]{CET18}; the lifting $\phi$ is a homotopy equivalence by \cite[Proposition 3.3(b)]{CET18}.
\end{proof}

\begin{cor}\label{GPcor}
Assume $S$ is a regular local ring, let $f\in S$ be nonzero, and set $R=S/(f)$. 
There is a triangulated equivalence between $\HF(\Prj{S},f)$ and the stable category of Gorenstein projective $R$-modules.
\end{cor}
\begin{proof}
Combine Theorem \ref{mainthm} with the equivalence between $\Ktac(\Prj{R})$ and the stable category of Gorenstein projective $R$-modules; see e.g., \cite[Example 3.10]{CET18}.
\end{proof}

There are dual results for injective factorizations:
\begin{prp}\label{gitomf}
Assume $S$ is a regular local ring, let $f\in S$ be nonzero, and set $R=S/(f)$. 
If $M$ is a Gorenstein injective $R$-module, then there exists an injective factorization $(I,d)\in \F(\Inj{S},f)$ with $\ker(d_1)=M$.
\end{prp}
\begin{proof}
Dual to the proof of Proposition \ref{gptomf}, where one instead uses \cite[Theorem 4.2]{BM10} in place of \cite[Theorem 4.1]{BM10} and \cite[6.2.6]{Chr00} in place of \cite[4.4.7]{Chr00}.
\end{proof}

\begin{thm}\label{mainthm_inj}
Assume $S$ is a regular local ring, let $f\in S$ be nonzero, and set $R=S/(f)$. 
There is a triangulated equivalence
\[\xymatrix{
\widetilde{\sfT}:\HF(\Inj{S},f)\ar[r]^-{\simeq} & \Ktac(\Inj{R})
}\]
given by the functor from Theorem \ref{fullyfaithful_div}.
\end{thm}
\begin{proof}
Similar to the proof of Theorem \ref{mainthm}; appeal instead to Theorem \ref{fullyfaithful_div} and Proposition \ref{gitomf}.
\end{proof}

\begin{cor}\label{GIcor}
Assume $S$ is a regular local ring, let $f\in S$ be nonzero, and set $R=S/(f)$. 
There is a triangulated equivalence between $\HF(\Inj{S},f)$ and the stable category of Gorenstein injective $R$-modules.
\end{cor}
\begin{proof}
Combine Theorem \ref{mainthm_inj} and the equivalence between $\Ktac(\Inj{R})$ and the stable category of Gorenstein injective $R$-modules; see \cite[Proposition 7.2]{Kra05}. 
\end{proof}

\section{An equivalence for flat-cotorsion factorizations}\label{equiv_fc}
\noindent
In this section, assume $S$ is a commutative noetherian ring. 
We give an equivalence in the case of the self-orthogonal category $\FC{S}$ of flat-cotorsion $S$-modules (that is, the category of $S$-modules that are both flat and cotorsion). 
The approach is similar to the previous section, but requires some extra care; in particular, we must establish a fact corresponding to the one from \cite{BM10} used above. 

Denote by $M^\wedge_\fp=\invlim(S/\fp^n\otimes_S M)$ the $\fp$-adic completion of an $S$-module $M$. By \cite{Eno84}, an $S$-module $M$ is flat-cotorsion if and only if it is isomorphic to a product over $\fp\in \Spec S$ of completions of free $S_\fp$-modules, that is, $M\cong \prod_{\fp\in \Spec S}(\bigoplus_{B(\fp)}S_\fp)^\wedge_\fp$ for some sets $B(\fp)$. 

\begin{lem}\label{FC_changerings}
Let $\pi:S\to R$ be a surjective ring homomorphism. Then we have an equality $R\otimes_S \FC{S}=\FC{R}$.
\end{lem}
\begin{proof}
First notice that for a flat-cotorsion $S$-module $\prod_{\fp\in \Spec S}(\bigoplus_{B(\fp)}S_\fp)^\wedge_\fp$, there is an isomorphism
$$\textstyle{R\otimes_S\left(\prod_{\fp\in \Spec S}(\bigoplus_{B(\fp)}S_\fp)^\wedge_\fp\right)\cong \prod_{\fp\in \Spec S}(\bigoplus_{B(\fp)}R_{\pi(\fp)})^\wedge_{\pi(\fp)}},$$
since $R$ is finitely presented as an $S$-module. It is now immediate that there is an inclusion $R\otimes_S\FC{S}\subseteq \FC{R}$. The other inclusion follows by observing that every flat-cotorsion $R$-module can be expressed in a form given by the right side of this isomorphism, since $\Spec R=\pi(\Spec S)$.
\end{proof}

The next lemma is needed in place of the change of rings facts for Gorenstein projective and Gorenstein injective dimensions from \cite{BM10}. As in \cite[Definition 4.3]{CET18}, refer to a $\FC{S}$-totally acyclic complex as a \emph{totally acyclic complex of flat-cotorsion $S$-modules} and a $\FC{S}$-Gorenstein module as a \emph{Gorenstein flat-cotorsion $S$-module}; see Definition \ref{tac_Gor}. Gorenstein flat-cotorsion $S$-modules are by \cite[Theorem 5.2]{CET18} precisely the modules that are both Gorenstein flat---that is, isomorphic to $\Cy{0}{F}$ for some \textbf{F}-totally acyclic complex $F$ of flat $S$-modules---and cotorsion.
\begin{lem}\label{gfcdim1}
Let $f\in S$ be a regular element and set $R=S/(f)$. Let $M$ be a Gorenstein flat-cotorsion $R$-module. There is an exact sequence of $S$-modules
\[\xymatrix{
0\ar[r] & M' \ar[r] & F \ar[r] & M\ar[r] & 0,
}\]
with $M'$ a Gorenstein flat-cotorsion $S$-module and $F$ a flat-cotorsion $S$-module.
\end{lem}
\begin{proof}
As $M$ is a Gorenstein flat-cotorsion $R$-module, there is a totally acyclic complex $T$ of flat-cotorsion $R$-modules such that $\Cy{0}{T}=M$. For each $i\in\ZZ$, we may find---because flat covers exist for all modules \cite{BEBE01}---a surjective flat cover $F_i\to \Cy{i}{T}$ over $S$; the kernel $K_i=\ker(F_i\to \Cy{i}{T})$ is cotorsion by Wakamatsu's Lemma \cite[Lemma 2.1.1]{Xu96}. In fact, since $\Cy{i}{T}$ is a cotorsion $R$-module, it is also a cotorsion $S$-module for each $i\in \ZZ$ by \cite[Proposition 3.3.3]{Xu96}, hence $F_i$ is flat-cotorsion for each $i\in \ZZ$. Indeed, $\Cy{i}{T}$ being a cotorsion $S$-module also yields $\Ext_S^1(F_{i-1},\Cy{i}{T})=0$; from this and the snake lemma we obtain, for each $i\in \ZZ$, the following commutative diagram with exact rows and columns:
\[\xymatrix{
 & 0 \ar[d] & 0 \ar[d] & 0 \ar[d]\\
0 \ar[r] & K_i \ar[r]\ar[d] & T_i' \ar[r]\ar[d] & K_{i-1} \ar[r]\ar[d] & 0\\
0 \ar[r] & F_i \ar[r]\ar[d] & F_i\oplus F_{i-1} \ar[r]\ar[d] & F_{i-1} \ar[r]\ar[d] & 0\\
0 \ar[r] & \Cy{i}{T} \ar[r]\ar[d] & T_i \ar[r]\ar[d] & \Cy{i-1}{T} \ar[r]\ar[d] & 0\\
 & 0 & 0 & 0
}\]
As $K_i$ and $K_{i-1}$ are cotorsion $S$-modules, so is $T_i'$. Additionally, as $T_i$ is a flat $R$-module, $\fd_S T_i=1$; see Remark \ref{pdidfd}. From \cite[2.4.F]{AF91}, we obtain that $T_i'$ is a flat $S$-module.

Now glue together the short exact sequences from the top rows of these diagrams to obtain an acyclic complex $T'$ of flat-cotorsion $S$-modules with $\Cy{i}{T'}=K_i$ for each $i\in \ZZ$. Fix a flat-cotorsion $S$-module $N$. Evidently, as each $K_i$ is cotorsion, we obtain $\Hom_S(N,T')$ is acyclic. Moreover, for each $i\in \ZZ$,
$$\Ext_S^1(K_i,N)\cong \Ext_S^2(\Cy{i}{T},N)\cong \Ext_R^1(\Cy{i}{T},R\otimes_SN)=0,$$
where the first isomorphism follows from the left vertical exact sequence in the diagram and the second follows from Lemma \ref{Rees}(i). For the last equality, note that as $N$ is a flat-cotorsion $S$-module, we have $R\otimes_S N$ is a flat-cotorsion $R$-module by Lemma \ref{FC_changerings}. It now follows that $\Hom_S(T',N)$ is also acyclic.  Thus $T'$ is a totally acyclic complex of flat-cotorsion $S$-modules. In particular, $\Cy{0}{T'}=K_0$ is a Gorenstein flat-cotorsion $S$-module, and the claim follows.
\end{proof}

\begin{prp}\label{gftomf}
Assume $S$ is a regular local ring, let $f\in S$ be nonzero, and set $R=S/(f)$. If $M$ is a Gorenstein flat-cotorsion $R$-module, then there exists a flat-cotorsion factorization $(F,d)\in \F(\FC{S},f)$ with $\coker(d_1)=M$.
\end{prp}

\begin{proof}
Let $M$ be a Gorenstein flat-cotorsion $R$-module. 
Lemma \ref{gfcdim1} yields an exact sequence
\[\xymatrix{
0\ar[r] & F_1 \ar[r]^{d_1} & F_0 \ar[r] & M\ar[r] & 0,
}\]
with $F_0$ a flat-cotorsion $S$-module and $F_1$ a Gorenstein flat-cotorsion $S$-module. By \cite[Theorem 5.2]{CET18}, $F_1$ is cotorsion and Gorenstein flat. As $S$ is regular, we have $\fd_SM<\infty$, hence $\fd_SM=\GFdim_SM\leq 1$ by \cite[5.2.8]{Chr00}. Thus $F_1$ is also flat \cite[2.4.F]{AF91}, hence flat-cotorsion. 

As in Proposition \ref{gptomf}, a standard construction applied to the sequence above provides a flat-cotorsion factorization $(F,d)$ of $f$ with $\coker (d_1)=M$.
\end{proof}

\begin{thm}\label{mainthm_fc}
Assume $S$ is a regular local ring, let $f\in S$ be nonzero, and set $R=S/(f)$. There is a triangulated equivalence
\[\xymatrix{
\sfT: \HF(\FC{S},f)\ar[r]^-{\simeq} & \Ktac(\FC{R})
}\]
given by the functor in Proposition \ref{functor}.
\end{thm}
\begin{proof}
Similar to the proof of Theorem \ref{mainthm}, using Proposition \ref{gftomf} in place of Proposition \ref{gptomf}, and Lemma \ref{FC_changerings} in place of Remark \ref{prj_inj_eq}.
\end{proof}

\begin{cor}\label{GFcor}
Assume $S$ is a regular local ring, let $f\in S$ be nonzero, and set $R=S/(f)$. There is a triangulated equivalence between $\HF(\FC{S},f)$ and the stable category of Gorenstein flat-cotorsion $R$-modules.
\end{cor}
\begin{proof}
This equivalence follows from Theorem \ref{mainthm_fc} and the triangulated equivalence between $\Ktac(\FC{R})$ and the stable category of Gorenstein flat-cotorsion $R$-modules given in \cite[Summary 5.7]{CET18}.
\end{proof}

One motivation for considering totally acyclic complexes of flat-cotorsion $R$-modules is their relation to the next analogue of the singularity category as described by Murfet and Salarian \cite{MS11}.  

The pure derived category of flat $S$-modules is defined as the Verdier quotient $\D(\Flat{S})=\K(\Flat{S})/\Kpac(\Flat{S})$ of the homotopy category of flat $S$-modules by its subcategory of pure acyclic complexes of flat $S$-modules. Neeman proves in \cite[Theorem 1.2]{Nee08} that $\D(\Flat{S})$ is equivalent to $\K(\Prj{S})$, and moreover, Murfet and Salarian show \cite[Lemma 4.22]{MS11} that $\DFtac(\Flat{S})$, the subcategory of $\D(\Flat{S})$ of {\bf F}-totally acyclic complexes, is equivalent to $\Ktac(\Prj{S})$, assuming that $S$ is a commutative noetherian ring having finite Krull dimension.
\begin{cor}\label{purederivedcor}
Assume $S$ is a regular local ring, let $f\in S$ be nonzero, and set $R=S/(f)$. There is a triangulated equivalence
\[\xymatrix{
\HF(\FC{S},f)\ar[r]^-{\simeq} &\DFtac(\Flat{R}).
}\]
\end{cor}
\begin{proof}
Combine Theorem \ref{mainthm_fc} and \cite[Summary 5.7]{CET18}.
\end{proof}

\addtocontents{toc}{\protect\setcounter{tocdepth}{0}}
\section*{Acknowledgments}
\noindent
We are grateful to Lars Winther Christensen and Mark Walker for many discussions related to the topics of this paper, and also to the anonymous referee for many helpful suggestions.

\vspace{-2mm}

\end{document}